\documentclass{amsart}
\usepackage{amssymb}
\usepackage{amsmath}
\usepackage{amsfonts}

\newtheorem{theorem}{Theorem}
\theoremstyle{plain}

\newtheorem{corollary}{Corollary}

\newtheorem{definition}{Definition}

\newtheorem{proposition}{Proposition}
\newtheorem{remark}{Remark}

\numberwithin{equation}{section}
\input{tcilatex}

\begin{document}
\title[Three Solutions To a Steklov Problem]{Three Solutions To a Steklov
Problem Involving The Weighted $p\left( .\right) $-Laplacian }
\author{Ismail AYDIN}
\address{Sinop University\\
Faculty of Arts and Sciences\\
Department of Mathematics\\
Sinop, TURKEY}
\email{iaydin@sinop.edu.tr}
\urladdr{}
\thanks{}
\author{Cihan UNAL}
\address{Assessment, Selection and Placement Center\\
Ankara, TURKEY}
\email{cihanunal88@gmail.com}
\urladdr{}
\thanks{}
\dedicatory{}
\thanks{}

\begin{abstract}
This paper is concerned with a nonlinear Steklov boundary-value problem
involving weighted $p\left( .\right) $-Laplacian. Using the Ricceri's
variational principle, we obtain the existence of at least three weak
solutions in double weighted variable exponent Sobolev space.
\end{abstract}

\maketitle

\section{Introduction}

The purpose of the present paper is to study the following Steklov problem
of the type%
\begin{equation}
\left\{ 
\begin{array}{cc}
\func{div}\left( a(x)\left\vert \nabla u\right\vert ^{p(x)-2}\nabla u\right)
=b(x)\left\vert u\right\vert ^{p(x)-2}u, & x\in \Omega \\ 
a(x)\left\vert \nabla u\right\vert ^{p(x)-2}\frac{\partial u}{\partial
\upsilon }=\lambda f\left( x,u\right) , & x\in \partial \Omega ,%
\end{array}%
\right.  \label{P}
\end{equation}%
where $\Omega \subset 
\mathbb{R}
^{N}$ $(N\geq 2)$ is a bounded smooth domain, $\frac{\partial u}{\partial
\upsilon }$ is the outer unit normal derivative on $\partial \Omega $, $%
\lambda >0$ is a real number, $p$ is a continuous function on $\overline{%
\Omega }$, i.e. $p\in C\left( \overline{\Omega }\right) $ with $\underset{%
y\in \overline{\Omega }}{\inf }p(y)>N$, the function $f:\partial \Omega
\times 
\mathbb{R}
\rightarrow 
\mathbb{R}
$ will be specified later, $a(x)$ and $b(x)$ are weight functions.

Nonlinear and variational problems involving $p(x)$-Laplacian operator has
attracted great attention in recent years. Because several physical
problems, such as electrorheological fluids, elastic mechanics, stationary
thermo-rheological viscous flows of non-Newtonian fluids, exponential growth
and image processing, can be modeled by such kind of equations, see \cite{Ha}%
, \cite{Ru}, \cite{Zh}. In recent years $p(x)$-Laplacian equations have
attracted great attention, see \cite{ay}, \cite{den}, \cite{fan}, \cite{fa}, 
\cite{mi}, \cite{un}. Moreover, Steklov problems involving $p(x)$-Laplacian
operator have been studied by many authors, see \cite{af}, \cite{ala}, \cite%
{al}, \cite{ba}, \cite{ali}, \cite{de}, \cite{hi}. In 2008, Deng \cite{de}
considered the eigenvalue of $p(x)$-Laplacian Steklov problem and proved the
existence of infinitely many eigenvalue sequences. Moreover, Mostofa et al. 
\cite{al}\ obtained the existence and multiplicity of solutions of the
nonlinear Steklov boundary-value problem using Ricceri's result in weighted
variable exponent Sobolev spaces.

In this paper, we define double weighted variable exponent Sobolev spaces
and give some basic properties of these spaces. We also obtain more general
results than above mentioned references, especially \cite{al}, under some
suitable conditions. Finally, we show that the problem (\ref{P}) has at
least three weak solutions due to the approach of Ricceri \cite{ri}.

\section{Notation and Preliminaries}

In this section, we give some definitions and basic informations about
weighted variable Lebesgue and Sobolev spaces to find out the solution of
the problem (\ref{P}). A normed space $\left( X,\left\Vert .\right\Vert
_{X}\right) $ is called a Banach function space (shortly BF-space), if
Banach space $\left( X,\left\Vert .\right\Vert _{X}\right) $ is continuously
embedded into $L_{loc}^{1}\left( \Omega \right) ,$ briefly $X\hookrightarrow
L_{loc}^{1}\left( \Omega \right) ,$ i.e. for any compact subset $K\subset
\Omega $ there is some constant $c_{K}>0$ such that $\left\Vert f\chi
_{K}\right\Vert _{L^{1}\left( \Omega \right) }\leq c_{K}\left\Vert
f\right\Vert _{X}$ for every $f\in X.$ Moreover, a normed space $X$ is
compactly embedded in a normed space $Y,$ briefly $X\hookrightarrow
\hookrightarrow Y,$ if $X\hookrightarrow Y$ and the identity operator $%
I:X\longrightarrow Y$ is compact, equivalently, $I$ maps every bounded
sequence $\left( x_{i}\right) _{i\in 
\mathbb{N}
}$ into a sequence $\left( I\left( x_{i}\right) \right) _{i\in 
\mathbb{N}
}$ that contains a subsequence converging in $Y.$ Suppose that $X$ and $Y$
are two Banach spaces and $X$ is reflexive. Then $I:X\longrightarrow Y$ is a
compact operator if and only if $I$ maps weakly convergent sequences in $X$
onto convergent sequences in $Y.$ More details can be found in \cite{ada}.
Suppose that $\Omega $ is a bounded open domain of $%
\mathbb{R}
^{N}$ with a smooth boundary $\partial \Omega $ and $p\in C_{+}\left( 
\overline{\Omega }\right) $, where 
\begin{equation*}
C_{+}\left( \overline{\Omega }\right) =\left\{ p\in C\left( \overline{\Omega 
}\right) :\inf_{x\in \overline{\Omega }}p(x)>1\right\} .
\end{equation*}%
For any $p\in C_{+}\left( \overline{\Omega }\right) $, we denote%
\begin{equation*}
1<p^{-}=\inf_{x\in \Omega }p(x)\leq p^{+}=\sup_{x\in \Omega }p(x)<\infty .
\end{equation*}%
Let $p\in C_{+}\left( \overline{\Omega }\right) $. The variable exponent
Lebesgue space $L^{p(.)}(\Omega )$ consists of all measurable functions $u$
such that $\varrho _{p(.)}(u)<\infty $ , equipped with the Luxemburg norm%
\begin{equation*}
\left\Vert u\right\Vert _{p(.)}=\inf \left\{ \lambda >0:\varrho
_{p(.)}\left( \frac{u}{\lambda }\right) \leq 1\right\} \text{,}
\end{equation*}%
where 
\begin{equation*}
\varrho _{p(.)}(u)=\int\limits_{\Omega }\left\vert u(x)\right\vert ^{p(x)}dx%
\text{.}
\end{equation*}%
The space $L^{p(.)}\left( \Omega \right) $ is a Banach space with respect to 
$\left\Vert .\right\Vert _{p(.)}$. If $p\left( .\right) =p$ is a constant
function, then the norm $\left\Vert .\right\Vert _{p(.)}$ coincides with the
usual Lebesgue norm $\left\Vert .\right\Vert _{p}$, see \cite{ko}. A
measurable and locally integrable function $a:\Omega \longrightarrow \left(
0,\infty \right) $ is called a weight function. Define the weighted variable
exponent Lebesgue space by%
\begin{equation*}
L_{a}^{p(.)}(\Omega )=\left\{ u\left\vert u:\Omega \longrightarrow 
\mathbb{R}
\text{ measurable and }\int\limits_{\Omega }\left\vert u(x)\right\vert
^{p(x)}a(x)dx<+\infty \right. \right\}
\end{equation*}%
with the Luxemburg norm 
\begin{equation*}
\left\Vert u\right\Vert _{p(.),a}=\inf \left\{ \tau >0:\varrho
_{p(.),a}\left( \frac{u}{\tau }\right) \leq 1\right\} ,
\end{equation*}%
where 
\begin{equation*}
\varrho _{p(.),a}(u)=\int\limits_{\Omega }\left\vert u(x)\right\vert
^{p(x)}a(x)dx\text{.}
\end{equation*}%
The space $L_{a}^{p(.)}(\Omega )$ is a Banach space with respect to $%
\left\Vert .\right\Vert _{p(.),a}.$ Moreover, $u\in L_{a}^{p(.)}(\Omega )$
if and only if $\left\Vert u\right\Vert _{p(.),a}=\left\Vert ua^{\frac{1}{%
p(.)}}\right\Vert _{p(.)}<\infty $. It is known that the relationships
between $\varrho _{p(.),a}$ and $\left\Vert .\right\Vert _{p(.),a}$ as%
\begin{equation*}
\min \left\{ \varrho _{p(.),a}(u)^{\frac{1}{p^{-}}},\varrho _{p(.),a}(u)^{%
\frac{1}{p^{+}}}\right\} \leq \left\Vert u\right\Vert _{p(.),a}\leq \max
\left\{ \varrho _{p(.),a}(u)^{\frac{1}{p^{-}}},\varrho _{p(.),a}(u)^{\frac{1%
}{p^{+}}}\right\}
\end{equation*}%
and%
\begin{equation*}
\min \left\{ \left\Vert u\right\Vert _{p(.),a}^{p^{-}},\left\Vert
u\right\Vert _{p(.),a}^{p^{+}}\right\} \leq \varrho _{p(.),a}(u)\leq \max
\left\{ \left\Vert u\right\Vert _{p(.),a}^{p^{-}},\left\Vert u\right\Vert
_{p(.),a}^{p^{+}}\right\}
\end{equation*}%
are satisfied. Also, if $0<C_{1}\leq a(x)$ for all $x\in \Omega $, then we
have $L_{a}^{p(.)}(\Omega )\hookrightarrow L^{p(.)}(\Omega ),$ since one
easily sees that 
\begin{equation*}
C_{1}\int\limits_{\Omega }\left\vert u(x)\right\vert ^{p(x)}dx\leq
\int\limits_{\Omega }\left\vert u(x)\right\vert ^{p(x)}a(x)dx
\end{equation*}%
and $C_{1}\left\Vert u\right\Vert _{p(.)}\leq \left\Vert u\right\Vert
_{p(.),a}$. Moreover, the dual space of $L_{a}^{p(.)}(\Omega )$ is $%
L_{a^{\ast }}^{q(.)}(\Omega )$, where $\frac{1}{p(.)}+\frac{1}{q(.)}=1$ and $%
a^{\ast }=a^{1-q\left( .\right) }=a^{-\frac{1}{p(.)-1}}.$ Let $a:\Omega
\longrightarrow \left( 0,\infty \right) $. In addition, the space $%
L_{a}^{p(.)}(\partial \Omega )$ can be defined by%
\begin{equation*}
L_{a}^{p(.)}(\partial \Omega )=\left\{ u\left\vert u:\partial \Omega
\longrightarrow 
\mathbb{R}
\text{ measurable and }\int\limits_{\partial \Omega }\left\vert
u(x)\right\vert ^{p(x)}a(x)d\sigma <+\infty \right. \right\}
\end{equation*}%
equipped with the Luxemburg norm, where $d\sigma $ is the measure on the
boundary. Then the space $L_{a}^{p(.)}(\partial \Omega )$ is a Banach space
with respect to $\left\Vert .\right\Vert _{p(.),a}$. If $a\in L^{\infty
}\left( \Omega \right) $, then we get $L_{a}^{p(.)}=L^{p(.)}.$

\begin{theorem}
(see \cite{ayd}) If $a^{-\frac{1}{p(.)-1}}\in L_{loc}^{1}\left( \Omega
\right) $, then $L_{a}^{p(.)}(\Omega )\hookrightarrow L_{loc}^{1}\left(
\Omega \right) \hookrightarrow D^{\prime }(\Omega )$, that is, every
function in $L_{a}^{p(.)}(\Omega )$ has distributional (weak) derivative,
where $D^{\prime }(\Omega )$ is distribution space.
\end{theorem}

\begin{remark}
If $a^{-\frac{1}{p(.)-1}}\notin L_{loc}^{1}\left( \Omega \right) $, then the
embedding $L_{a}^{p(.)}(\Omega )\hookrightarrow L_{loc}^{1}\left( \Omega
\right) $ need not hold.
\end{remark}

\begin{definition}
Let $a^{-\frac{1}{p(.)-1}}\in L_{loc}^{1}\left( \Omega \right) $. We set the
weighted variable exponent Sobolev space $W_{a}^{k,p(.)}\left( \Omega
\right) $ by%
\begin{equation*}
W_{a}^{k,p(.)}\left( \Omega \right) =\left\{ u\in L_{a}^{p(.)}(\Omega
):D^{\alpha }u\in L_{a}^{p(.)}(\Omega ),0\leq \left\vert \alpha \right\vert
\leq k\right\}
\end{equation*}%
equipped with the norm 
\begin{equation*}
\left\Vert u\right\Vert _{k,p(.),a}=\sum\limits_{0\leq \left\vert \alpha
\right\vert \leq k}\left\Vert D^{\alpha }u\right\Vert _{p(.),a}
\end{equation*}%
where $\alpha \in 
\mathbb{N}
_{0}^{N}$ is a multi-index, $\left\vert \alpha \right\vert =\alpha
_{1}+\alpha _{2}+...+\alpha _{N}$ and $D^{\alpha }=\frac{\partial
^{\left\vert \alpha \right\vert }}{\partial _{x_{1}}^{\alpha
_{1}}...\partial _{x_{N}}^{\alpha _{N}}}$. It is known that $%
W_{a}^{k,p(.)}\left( \Omega \right) $ is a reflexive Banach space. In
particular, the space $W_{a}^{1,p(.)}\left( \Omega \right) $ is defined by 
\begin{equation*}
W_{a}^{1,p(.)}\left( \Omega \right) =\left\{ u\in L_{a}^{p(.)}(\Omega
):\left\vert \nabla u\right\vert \in L_{a}^{p(.)}(\Omega )\right\} .
\end{equation*}%
The function $\varrho _{1,p(.),a}:W_{a}^{1,p(.)}\left( \Omega \right)
\longrightarrow \left[ 0,\infty \right) $ is shown as $\varrho
_{1,p(.),a}(u)=\varrho _{p(.),a}(u)+\varrho _{p(.),a}\left( \nabla u\right) $%
. Also, the norm $\left\Vert u\right\Vert _{1,p(.),a}=\left\Vert
u\right\Vert _{p(.),a}+\left\Vert \nabla u\right\Vert _{p(.),a}$ makes the
space $W_{a}^{1,p(.)}\left( \Omega \right) $ a Banach space.
\end{definition}

Let $a^{-\frac{1}{p(.)-1}}\in L_{loc}^{1}\left( \Omega \right) $ and $b^{-%
\frac{1}{p(.)-1}}\in L_{loc}^{1}\left( \Omega \right) $. The double weighted
variable exponent Sobolev space $W_{a,b}^{1,p(.)}\left( \Omega \right) $ is
defined by%
\begin{equation*}
W_{a,b}^{1,p(.)}\left( \Omega \right) =\left\{ u\in L_{b}^{p(.)}(\Omega
):\left\vert \nabla u\right\vert \in L_{a}^{p(.)}(\Omega )\right\}
\end{equation*}%
equipped with the norm%
\begin{equation*}
\left\Vert u\right\Vert _{1,p(.),a,b}=\left\Vert \nabla u\right\Vert
_{p(.),a}+\left\Vert u\right\Vert _{p(.),b}.
\end{equation*}%
Since $a^{-\frac{1}{p(.)-1}}\in L_{loc}^{1}\left( \Omega \right) $ and $b^{-%
\frac{1}{p(.)-1}}\in L_{loc}^{1}\left( \Omega \right) ,$ then it can be seen
that $L_{a}^{p(.)}(\Omega )\hookrightarrow L_{loc}^{1}\left( \Omega \right) $
and $L_{b}^{p(.)}(\Omega )\hookrightarrow L_{loc}^{1}\left( \Omega \right) $%
. Therefore, the double weighted variable exponent Sobolev space $%
W_{a,b}^{1,p(.)}\left( \Omega \right) $ is well-defined. The dual space of $%
W_{a,b}^{1,p(.)}\left( \Omega \right) $ is $W_{a^{\ast },b^{\ast
}}^{-1,q(.)}\left( \Omega \right) $, where $\frac{1}{p(.)}+\frac{1}{q(.)}=1$
and $a^{\ast }=a^{-\frac{1}{p(.)-1}},$ $b^{\ast }=b^{-\frac{1}{p(.)-1}}$.
Moreover, the space $W_{a,b}^{1,p(.)}\left( \Omega \right) $ is a separable
and reflexive Banach space.

\begin{proposition}
\label{pro 4}(see \cite{li}) Let $I(u)=\int\limits_{\Omega }\left(
a(x)\left\vert \nabla u\right\vert ^{p(x)}+b(x)\left\vert u\right\vert
^{p(x)}\right) dx$. For all $u\in W_{a,b}^{1,p(.)}\left( \Omega \right) $,

\begin{enumerate}
\item[\textit{(i)}] if $\left\Vert u\right\Vert _{1,p(.),a,b}\geq 1,$ then
the inequality $\left\Vert u\right\Vert _{1,p(.),a,b}^{p^{-}}\leq I(u)\leq
\left\Vert u\right\Vert _{1,p(.),a,b}^{p^{+}}$ is satisfied.

\item[\textit{(ii)}] if $\left\Vert u\right\Vert _{1,p(.),a,b}\leq 1,$ then
the inequality $\left\Vert u\right\Vert _{1,p(.),a,b}^{p^{+}}\leq I(u)\leq
\left\Vert u\right\Vert _{1,p(.),a,b}^{p^{-}}$ is satisfied.
\end{enumerate}
\end{proposition}

The following a compact embedding theorem of $W_{a,b}^{1,p(.)}\left( \Omega
\right) $ into $C\left( \overline{\Omega }\right) $ plays an important role
in this paper. For the proof, we use the method in \cite[Theorem 2.11]{ki}.

\begin{theorem}
\label{teo5}Let $a^{-\alpha \left( .\right) }\in L^{1}\left( \Omega \right) $
with $\alpha \left( x\right) \in \left( \frac{N}{p(x)},\infty \right) \cap %
\left[ \frac{1}{p(x)-1},\infty \right) $. If we define the variable exponent 
$p_{\ast }(x)=\frac{\alpha \left( x\right) p(x)}{\alpha \left( x\right) +1}$
with $N<p_{\ast }^{-}$, then we have the compact embedding $%
W_{a,b}^{1,p(.)}\left( \Omega \right) \hookrightarrow \hookrightarrow
C\left( \overline{\Omega }\right) .$
\end{theorem}

\begin{proof}
First we will show that the continuous embedding $W_{a,b}^{1,p(.)}\left(
\Omega \right) \hookrightarrow W^{1,p_{\ast }(.)}\left( \Omega \right) $ is
valid. Let $u\in W_{a,b}^{1,p(.)}\left( \Omega \right) .$ Then we write that 
$u\in L_{b}^{p(.)}(\Omega )$ and $\nabla u\in L_{a}^{p(.)}(\Omega )$. Using H%
\"{o}lder's inequality with $q(x)=\frac{p(x)}{p_{\ast }(x)}=\frac{\alpha
\left( x\right) +1}{\alpha \left( x\right) }$ and $q^{\prime }(x)=\alpha
\left( x\right) +1$, we have%
\begin{eqnarray*}
\int\limits_{\Omega }\left\vert \nabla u\left( x\right) \right\vert
^{p_{\ast }\left( x\right) }dx &=&\int\limits_{\Omega }\left\vert \nabla
u\left( x\right) \right\vert ^{\frac{\alpha \left( x\right) p(x)}{\alpha
\left( x\right) +1}}dx \\
&=&\int\limits_{\Omega }\left\vert \nabla u\left( x\right) \right\vert ^{%
\frac{\alpha \left( x\right) p(x)}{\alpha \left( x\right) +1}}a^{\frac{%
\alpha \left( x\right) }{\alpha \left( x\right) +1}}(x)a^{-\frac{\alpha
\left( x\right) }{\alpha \left( x\right) +1}}(x)dx \\
&\leq &2\left\Vert a^{\frac{\alpha \left( .\right) }{\alpha \left( .\right)
+1}}\left\vert \nabla u\right\vert ^{\frac{\alpha \left( .\right) p(.)}{%
\alpha \left( .\right) +1}}\right\Vert _{\frac{\alpha \left( .\right) +1}{%
\alpha \left( .\right) }}\left\Vert a^{-\frac{\alpha \left( .\right) }{%
\alpha \left( .\right) +1}}\right\Vert _{\alpha \left( .\right) +1}.
\end{eqnarray*}

It is well known that $\varrho _{r(.)}(u)<\infty $ if and only if $%
\left\Vert u\right\Vert _{r(.)}<\infty $. Since $a^{-\alpha \left( .\right)
}\in L^{1}\left( \Omega \right) $, then $\left\Vert a^{-\frac{\alpha \left(
.\right) }{\alpha \left( .\right) +1}}\right\Vert _{\alpha \left( .\right)
+1}<C_{2}<\infty $. Thus, we obtain%
\begin{equation}
\int\limits_{\Omega }\left\vert \nabla u\left( x\right) \right\vert
^{p_{\ast }\left( x\right) }dx\leq C_{3}\left\Vert a^{\frac{\alpha \left(
.\right) }{\alpha \left( .\right) +1}}\left\vert \nabla u\right\vert ^{\frac{%
\alpha \left( .\right) p(.)}{\alpha \left( .\right) +1}}\right\Vert _{\frac{%
\alpha \left( .\right) +1}{\alpha \left( .\right) }}.  \label{1}
\end{equation}%
In general, we can suppose that $\int\limits_{\Omega }\left\vert \nabla
u\left( x\right) \right\vert ^{p_{\ast }\left( x\right) }dx>1.$ Because if $%
\int\limits_{\Omega }\left\vert \nabla u\left( x\right) \right\vert
^{p_{\ast }\left( x\right) }dx\leq 1$, then $\nabla u\in L^{p_{\ast }\left(
.\right) }(\Omega )$ and $u\in L^{p_{\ast }\left( .\right) }(\Omega )$ due
to $p_{\ast }(.)<p(.)$. Thus we have $u\in W^{1,p_{\ast }(.)}\left( \Omega
\right) $. If $\int\limits_{\Omega }a(x)\left\vert \nabla u\right\vert
^{p(x)}dx\leq 1$, then by (\ref{1}) and Proposition \ref{pro 4} we have%
\begin{eqnarray*}
\left\Vert \nabla u\right\Vert _{p_{\ast }}^{\frac{\alpha ^{-}p^{-}}{\alpha
^{+}+1}} &\leq &C_{3}\left\Vert a^{\frac{\alpha \left( .\right) }{\alpha
\left( .\right) +1}}\left\vert \nabla u\right\vert ^{\frac{\alpha \left(
.\right) p(.)}{\alpha \left( .\right) +1}}\right\Vert _{\frac{\alpha \left(
.\right) +1}{\alpha \left( .\right) }} \\
&\leq &C_{3}\left( \int\limits_{\Omega }\left\vert \nabla u\left( x\right)
\right\vert ^{p\left( x\right) }a\left( x\right) dx\right) ^{\frac{\alpha
^{-}}{\alpha ^{+}+1}}\leq C_{3}\left\Vert \nabla u\right\Vert _{p(.),a}^{%
\frac{\alpha ^{-}p^{-}}{\alpha ^{+}+1}}.
\end{eqnarray*}%
This follows%
\begin{equation}
\left\Vert \nabla u\right\Vert _{p_{\ast }}\leq C_{4}\left\Vert \nabla
u\right\Vert _{p(.),a},  \label{2}
\end{equation}%
where $C_{4}=C_{3}^{\frac{\alpha ^{+}+1}{\alpha ^{-}p^{-}}}>0$. On the other
hand, if $\int\limits_{\Omega }a(x)\left\vert \nabla u\right\vert
^{p(x)}dx>1,$ then by (\ref{1}) and Proposition \ref{pro 4} we obtain%
\begin{eqnarray*}
\left\Vert \nabla u\right\Vert _{p_{\ast }}^{\frac{\alpha ^{-}p^{-}}{\alpha
^{+}+1}} &\leq &C_{3}\left\Vert a^{\frac{\alpha \left( .\right) }{\alpha
\left( .\right) +1}}\left\vert \nabla u\right\vert ^{\frac{\alpha \left(
.\right) p(.)}{\alpha \left( .\right) +1}}\right\Vert _{\frac{\alpha \left(
.\right) +1}{\alpha \left( .\right) }} \\
&\leq &C_{3}\left( \int\limits_{\Omega }\left\vert \nabla u\left( x\right)
\right\vert ^{p\left( x\right) }a\left( x\right) dx\right) ^{\frac{\alpha
^{+}}{\alpha ^{-}+1}}\leq C_{3}\left\Vert \nabla u\right\Vert _{p(.),a}^{%
\frac{\alpha ^{+}p^{+}}{\alpha ^{-}+1}},
\end{eqnarray*}%
or equivalently%
\begin{equation}
\left\Vert \nabla u\right\Vert _{p_{\ast }}\leq C_{4}\left\Vert \nabla
u\right\Vert _{p(.),a}^{\beta },  \label{3}
\end{equation}%
where $\beta =\frac{\alpha ^{+}p^{+}}{\alpha ^{-}+1}.\frac{\alpha ^{+}+1}{%
\alpha ^{-}p^{-}}$. If we consider the (\ref{2}) and (\ref{3}), then we have 
$\nabla u\in L^{p_{\ast }\left( .\right) }\left( \Omega \right) .$
Therefore, we have $u\in W^{1,p_{\ast }(.)}\left( \Omega \right) .$ Hence,
the inclusion $W_{a,b}^{1,p(.)}\left( \Omega \right) \subset W^{1,p_{\ast
}(.)}\left( \Omega \right) $ is satisfied. Using the Banach closed graph
theorem, we get%
\begin{equation*}
W_{a,b}^{1,p(.)}\left( \Omega \right) \hookrightarrow W^{1,p_{\ast
}(.)}\left( \Omega \right) .
\end{equation*}%
Since $p_{\ast }^{-}>N$ and $W^{1,p_{\ast }(.)}\left( \Omega \right)
\hookrightarrow W^{1,p_{\ast }^{-}}\left( \Omega \right) $ it follows that $%
W^{1,p_{\ast }^{-}}\left( \Omega \right) \hookrightarrow \hookrightarrow
C\left( \overline{\Omega }\right) $ and $W_{a,b}^{1,p(.)}\left( \Omega
\right) \hookrightarrow \hookrightarrow C\left( \overline{\Omega }\right) $.
This completes the proof.
\end{proof}

\begin{corollary}
\label{cor6}Since $W_{a,b}^{1,p(.)}\left( \Omega \right) \hookrightarrow
\hookrightarrow C\left( \overline{\Omega }\right) $, then there exists a $%
C_{5}>0$ such that 
\begin{equation*}
\left\Vert u\right\Vert _{\infty }\leq C_{5}\left\Vert u\right\Vert
_{1,p(.),a,b}
\end{equation*}%
for any $u\in W_{a,b}^{1,p(.)}\left( \Omega \right) $, where $\left\Vert
u\right\Vert _{\infty }=\underset{x\in \overline{\Omega }}{\sup }u(x)$ for $%
u\in C\left( \overline{\Omega }\right) $.
\end{corollary}

For $A\subset \overline{\Omega }$, denote by $\theta ^{-}(A)=\underset{x\in A%
}{\inf }\theta (x)$ and $\theta ^{+}(A)=\underset{x\in A}{\sup }\theta (x)$.
For any $x\in \partial \Omega $ and $r\in C\left( \partial \Omega ,%
\mathbb{R}
\right) $ with $r^{-}=\underset{x\in \partial \Omega }{\inf }r(x)>1,$ we
define%
\begin{equation*}
\theta ^{\partial }\left( x\right) =\left( \theta (x)\right) ^{\partial
}=\left\{ 
\begin{array}{cc}
\frac{\left( N-1\right) \theta (x)}{N-\theta (x)}, & \text{if }\theta \left(
x\right) <N, \\ 
\infty , & \text{if }\theta \left( x\right) \geq N,%
\end{array}%
\right.
\end{equation*}%
\begin{equation*}
\theta _{r(x)}^{\partial }\left( x\right) =\frac{r(x)-1}{r(x)}\theta
^{\partial }\left( x\right) .
\end{equation*}

\begin{theorem}
\label{teo7}(see \cite{de}) Assume that the boundary of $\Omega $ possesses
the cone property and $\theta \in C\left( \overline{\Omega }\right) $ with $%
\theta ^{-}>1$. Suppose that $a\in L^{r(.)}(\partial \Omega )$, $r\in
C\left( \partial \Omega \right) $ with $r(x)>\frac{\theta ^{\partial }\left(
x\right) }{\theta ^{\partial }\left( x\right) -1}$ for all $x\in \partial
\Omega $. If $q\in C\left( \partial \Omega \right) $ and $1\leq q(x)<\theta
_{r(x)}^{\partial }\left( x\right) $ for all $x\in \partial \Omega $, then
there is a compact embedding from $W^{1,\theta (.)}\left( \Omega \right) $
into $L_{a}^{q(.)}(\partial \Omega )$. In particular, there is a compact
embedding from $W^{1,\theta (.)}\left( \Omega \right) $ into $%
L^{q(.)}(\partial \Omega )$, where $1\leq q(x)<\theta ^{\partial }\left(
x\right) $ for all $x\in \partial \Omega $.
\end{theorem}

\begin{corollary}
\label{cor8}All conditions in Theorem \ref{teo5} and Theorem \ref{teo7} are
satisfied. If $q\in C\left( \partial \Omega \right) $ and $1\leq
q(x)<p_{\ast ,r(x)}^{\partial }\left( x\right) $, for all $x\in \partial
\Omega $, then we have $W^{1,p_{\ast }(.)}\left( \Omega \right)
\hookrightarrow \hookrightarrow L_{a}^{q(.)}(\partial \Omega )$. This yields
that $W_{a,b}^{1,p(.)}\left( \Omega \right) \hookrightarrow \hookrightarrow
L_{a}^{q(.)}(\partial \Omega )$. Moreover, we obtain $W_{a,b}^{1,p(.)}\left(
\Omega \right) \hookrightarrow W^{1,p_{\ast }(.)}\left( \Omega \right)
\hookrightarrow L^{q(.)}(\partial \Omega )$ for $1\leq q(x)<p_{\ast
}^{\partial }\left( x\right) $ for all $x\in \partial \Omega $.
\end{corollary}

\begin{theorem}
\label{teo9}(see \cite{ri}) Let $X$ be a separable and reflexive real Banach
space; $\Phi :X\rightarrow 
\mathbb{R}
$ a continuously G\^{a}teaux differentiable and sequentially weakly lower
semicontinuous functional whose G\^{a}teaux derivative admits a continuous
inverse on $X^{\ast }$; $\Psi :X\rightarrow 
\mathbb{R}
$ a continuously G\^{a}teaux differentiable functional whose G\^{a}teaux
derivative is compact. Assume that

\begin{enumerate}
\item[\textit{(i)}] $\underset{\left\Vert u\right\Vert \rightarrow \infty }{%
\lim }\left( \Phi (u)+\lambda \Psi (u)\right) =\infty $ for all $\lambda >0,$

\item[\textit{(ii)}] there are $r\in 
\mathbb{R}
$ and $u_{0},u_{1}\in X$ such that $\Phi (u_{0})<r<\Phi (u_{1}),$

\item[\textit{(iii)}] $\underset{u\in \Phi ^{-1}\left( \left( -\infty ,r%
\right] \right) }{\inf }\Psi (u)>\frac{\left( \Phi (u_{1})-r\right) \Psi
(u_{0})+\left( r-\Phi (u_{0})\right) \Psi (u_{1})}{\Phi (u_{1})-\Phi (u_{0})}%
.$

Then there exist an open interval $\Lambda \subset \left( 0,\infty \right) $
and a positive constant $\rho >0$ such that for any $\lambda \in \Lambda $
the equation $\Phi ^{\prime }\left( u\right) +\lambda \Psi ^{\prime }\left(
u\right) =0$ has at least three solutions in $X$ whose norms are less than $%
\rho $.
\end{enumerate}
\end{theorem}

\begin{proposition}
\label{pro10}(see \cite{li}) Let us consider the functional $\Phi
(u)=\int\limits_{\Omega }\frac{1}{p(x)}\left( a(x)\left\vert \nabla
u\right\vert ^{p(x)}+b(x)\left\vert u\right\vert ^{p(x)}\right) dx$ for all $%
u\in W_{a,b}^{1,p(.)}\left( \Omega \right) $. Then, we have

\begin{enumerate}
\item[\textit{(i)}] $\Phi :W_{a,b}^{1,p(.)}\left( \Omega \right)
\longrightarrow 
\mathbb{R}
$ is sequentially weakly lower semicontinuous and $\Phi \in C^{1}\left(
W_{a,b}^{1,p(.)}\left( \Omega \right) ,%
\mathbb{R}
\right) $. Moreover, the derivative operator $\Phi ^{\prime }$ of $\Phi $
define as%
\begin{equation*}
\left\langle \Phi ^{\prime }\left( u\right) ,v\right\rangle
=\int\limits_{\Omega }\left( a\left( x\right) \left\vert \nabla u\right\vert
^{p\left( x\right) -2}\nabla u\nabla v+b\left( x\right) \left\vert
u\right\vert ^{p\left( x\right) -2}uv\right) dx
\end{equation*}%
for all $u,v\in W_{a,b}^{1,p(.)}\left( \Omega \right) $.

\item[\textit{(ii)}] $\Phi ^{\prime }:W_{a,b}^{1,p(.)}\left( \Omega \right)
\longrightarrow W_{a^{\ast },b^{\ast }}^{-1,q(.)}\left( \Omega \right) $ is
a continuous, bounded and strictly monotone operator.

\item[\textit{(iii)}] $\Phi ^{\prime }$ is a mapping of type $\left(
S_{+}\right) ,$ i.e., if $u_{n}\rightharpoonup u$ in $W_{a,b}^{1,p(.)}\left(
\Omega \right) $ and $\underset{n\longrightarrow \infty }{\lim \sup }%
\left\langle \Phi ^{\prime }\left( u_{n}\right) -\Phi ^{\prime }\left(
u\right) ,u_{n}-u\right\rangle \leq 0$, then $u_{n}\longrightarrow u$ in $%
W_{a,b}^{1,p(.)}\left( \Omega \right) .$

\item[\textit{(iv)}] $\Phi ^{\prime }:W_{a,b}^{1,p(.)}\left( \Omega \right)
\longrightarrow W_{a^{\ast },b^{\ast }}^{-1,q(.)}\left( \Omega \right) $ is
a homeomorphism.
\end{enumerate}
\end{proposition}

\section{Main Results}

In this paper, we assume that $a^{-\frac{1}{p(.)-1}}\in L_{loc}^{1}\left(
\Omega \right) $ and $b^{-\frac{1}{p(.)-1}}\in L_{loc}^{1}\left( \Omega
\right) $, $a^{-\alpha \left( .\right) }\in L^{1}\left( \Omega \right) $
with $\alpha \left( x\right) \in \left( \frac{N}{p(x)},\infty \right) \cap %
\left[ \frac{1}{p(x)-1},\infty \right) $ and $N<p_{\ast }^{-}$. Moreover,
the function $f:$ $\partial \Omega \times 
\mathbb{R}
\rightarrow 
\mathbb{R}
$ is a Carath\'{e}odory function and satisfies the following conditions:

\begin{enumerate}
\item[$(F1)$] $\left\vert f(x,t)\right\vert \leq k(x)+c\left\vert
t\right\vert ^{s(x)-1}$ for all $(x,t)\in \partial \Omega \times 
\mathbb{R}
$ where $k(x)\in L^{\frac{s(x)}{s(x)-1}}(\partial \Omega )$, $k(x)\geq 0$
and $s(x)\in C_{+}\left( \partial \Omega \right) $, $1<s^{-}=\underset{x\in 
\overline{\Omega }}{\inf }s(x)\leq s^{+}=\underset{x\in \overline{\Omega }}{%
\sup }s(x)<p^{-}$ with $s(x)<p_{\ast }^{\partial }\left( x\right) $, for all 
$x\in \partial \Omega .$

\item[$(F2)$] \textit{(i)} $f(x,t)<0$ for all $(x,t)\in \partial \Omega
\times 
\mathbb{R}
,$ when $\left\vert t\right\vert \in \left( 0,1\right) $,

\textit{(ii)} $f(x,t)\geq M>0,$ when $\left\vert t\right\vert \in \left(
t_{0},\infty \right) $, $t_{0}>1$.
\end{enumerate}

\begin{definition}
We call that $f\in W_{a,b}^{1,p(.)}\left( \Omega \right) $ is a weak
solution of the problem (\ref{P}) if%
\begin{equation*}
\int\limits_{\Omega }a\left( x\right) \left\vert \nabla u\right\vert
^{p\left( x\right) -2}\nabla u\nabla vdx+\int\limits_{\Omega }b\left(
x\right) \left\vert u\right\vert ^{p\left( x\right) -2}uvdx-\lambda
\int\limits_{\partial \Omega }f(x,u)vd\sigma =0
\end{equation*}%
for all $v\in W_{a,b}^{1,p(.)}\left( \Omega \right) .$ The corresponding
energy functional of the problem (\ref{P})%
\begin{equation*}
J(u)=\Phi (u)+\lambda \Psi (u)
\end{equation*}%
where the functionals $\Phi ,\Psi $ from $W_{a,b}^{1,p(.)}\left( \Omega
\right) $ into $%
\mathbb{R}
$ as $\Phi (u)=\int\limits_{\Omega }\frac{1}{p(x)}\left( a(x)\left\vert
\nabla u\right\vert ^{p(x)}+b(x)\left\vert u\right\vert ^{p(x)}\right) dx$, $%
\Psi (u)=-\int\limits_{\partial \Omega }F(x,u)d\sigma $ and $%
F(x,t)=\int\limits_{0}^{t}f(x,y)dy$.

By \cite[Proposition 3.1]{mih} and Proposition \ref{pro 4}, we get $\Phi
,\Psi \in C^{1}\left( W_{a,b}^{1,p(.)}\left( \Omega \right) ,%
\mathbb{R}
\right) $ with the derivatives given by%
\begin{eqnarray*}
\left\langle \Phi ^{\prime }\left( u\right) ,v\right\rangle
&=&\int\limits_{\Omega }\left( a\left( x\right) \left\vert \nabla
u\right\vert ^{p\left( x\right) -2}\nabla u\nabla v+b\left( x\right)
\left\vert u\right\vert ^{p\left( x\right) -2}uv\right) dx, \\
\left\langle \Psi ^{\prime }\left( u\right) ,v\right\rangle
&=&-\int\limits_{\partial \Omega }f(x,u)vd\sigma ,
\end{eqnarray*}%
for any $u,v\in W_{a,b}^{1,p(.)}\left( \Omega \right) .$ Since $\Phi
^{\prime }:W_{a,b}^{1,p(.)}\left( \Omega \right) \longrightarrow W_{a^{\ast
},b^{\ast }}^{-1,q(.)}\left( \Omega \right) $ is a homeomorphism by
Proposition \ref{pro 4}, it is obvious that $\left( \Phi ^{\prime }\right)
^{-1}:W_{a^{\ast },b^{\ast }}^{-1,q(.)}\left( \Omega \right) \longrightarrow
W_{a,b}^{1,p(.)}\left( \Omega \right) $ exists and continuous. Moreover, $%
\Psi ^{\prime }:W_{a,b}^{1,p(.)}\left( \Omega \right) \longrightarrow
W_{a^{\ast },b^{\ast }}^{-1,q(.)}\left( \Omega \right) $ is completely
continuous due to the assumption $(F1)$ in \cite[Theorem 2.9]{al}, which
implies $\Psi ^{\prime }:W_{a,b}^{1,p(.)}\left( \Omega \right)
\longrightarrow W_{a^{\ast },b^{\ast }}^{-1,q(.)}\left( \Omega \right) $ is
compact.

We note that the operator $J$ is a $C^{1}\left( W_{a,b}^{1,p(.)}\left(
\Omega \right) ,%
\mathbb{R}
\right) $ functional and the critical points of $J$ are weak solutions of
the problem (\ref{P}).
\end{definition}

Now, we are ready to give our main result.

\begin{theorem}
If the conditions $(F1)$ and $(F2)$ are valid, then there exist an open
interval $\Lambda \subset \left( 0,\infty \right) $ and a positive constant $%
\rho >0$ such that for any $\lambda \in \Lambda $, the problem (\ref{P}) has
at least three solutions in $W_{a,b}^{1,p(.)}\left( \Omega \right) $ whose
norms are less than $\rho $.
\end{theorem}

\begin{proof}
To prove this theorem , we first verify the condition $(i)$ of Theorem \ref%
{teo9}. In fact, by Proposition \ref{pro10} we have%
\begin{eqnarray}
\Phi (u) &\geq &\frac{1}{p^{+}}\int\limits_{\Omega }\left( a(x)\left\vert
\nabla u\right\vert ^{p(x)}+b(x)\left\vert u\right\vert ^{p(x)}\right) dx 
\notag \\
&=&\frac{1}{p^{+}}I(u)\geq \frac{1}{p^{+}}\left\Vert u\right\Vert
_{1,p(.),a,b}^{p^{-}}  \label{5}
\end{eqnarray}%
for any $u\in W_{a,b}^{1,p(.)}\left( \Omega \right) $ with $\left\Vert
u\right\Vert _{1,p(.),a,b}>1$.

On the other hand, by $(F1)$ and the H\"{o}lder inequality, we get 
\begin{eqnarray}
-\Psi (u) &=&\int\limits_{\partial \Omega }F(x,u)d\sigma
=\int\limits_{\partial \Omega }\left( \int\limits_{0}^{u(x)}f(x,t)dt\right)
d\sigma  \notag \\
&\leq &\int\limits_{\partial \Omega }\left( k(x)\left\vert u(x)\right\vert +%
\frac{c}{s(x)}\left\vert u(x)\right\vert ^{s(x)}\right) d\sigma  \notag \\
&\leq &2\left\Vert k\right\Vert _{\frac{s(.)}{s(.)-1},\partial \Omega
}\left\Vert u\right\Vert _{s(.),\partial \Omega }+\frac{c}{s^{-}}%
\int\limits_{\partial \Omega }\left\vert u(x)\right\vert ^{s(x)}d\sigma .
\label{6}
\end{eqnarray}%
By Corollary \ref{cor8}, it is obtained that $W_{a,b}^{1,p(.)}\left( \Omega
\right) \hookrightarrow L^{s(.)}(\partial \Omega )$ and%
\begin{equation}
\int\limits_{\partial \Omega }\left\vert u(x)\right\vert ^{s(x)}d\sigma \leq
\max \left\{ \left\Vert u\right\Vert _{s(.),\partial \Omega
}^{s^{-}},\left\Vert u\right\Vert _{s(.),\partial \Omega }^{s^{+}}\right\}
\leq C_{6}\left\Vert u\right\Vert _{1,p(.),a,b}^{s^{+}}.  \label{7}
\end{equation}%
If we use (\ref{6}) and (\ref{7}), then we get%
\begin{equation}
-\Psi (u)\leq C_{7}\left\Vert k\right\Vert _{\frac{s(.)}{s(.)-1},\partial
\Omega }\left\Vert u\right\Vert _{1,p(.),a,b}+\frac{c}{s^{-}}C_{6}\left\Vert
u\right\Vert _{1,p(.),a,b}^{s^{+}}.  \label{8}
\end{equation}%
For any $\lambda >0$ we can obtain%
\begin{equation*}
\Phi (u)+\lambda \Psi (u)\geq \frac{1}{p^{+}}\left\Vert u\right\Vert
_{1,p(.),a,b}^{p^{-}}-\lambda C_{7}\left\Vert k\right\Vert _{\frac{s(.)}{%
s(.)-1},\partial \Omega }\left\Vert u\right\Vert _{1,p(.),a,b}-\frac{c}{s^{-}%
}\lambda C_{6}\left\Vert u\right\Vert _{1,p(.),a,b}^{s^{+}}
\end{equation*}%
by (\ref{5}) and (\ref{8}). Since $1<s^{+}<p^{-}$, then $\underset{%
\left\Vert u\right\Vert _{1,p(.),a,b}\rightarrow \infty }{\lim }\left( \Phi
(u)+\lambda \Psi (u)\right) =\infty $ for all $\lambda >0$ and \textit{(i)}
is verified.

By $\frac{\partial F(x,t)}{\partial t}=f(x,t)$ and $(F2)$, it is obtained
that $F(x,t)$ is increasing for $t\in \left( t_{0},\infty \right) $ and
decreasing for $t\in \left( 0,1\right) $, uniformly with respect to $x\in
\partial \Omega $, and $F(x,0)=0$. In addition, $F(x,t)\rightarrow \infty $
when $t\rightarrow \infty $ due to $F(x,t)\geq Mt$ uniformly for $x$. Then
there exists a real number $\delta >t_{0}$ such that 
\begin{equation}
F(x,t)\geq 0=F(x,0)\geq F(x,\tau ),\text{ \ for all }x\in \partial \Omega ,%
\text{ }t>\delta ,\text{ }\tau \in \left( 0,1\right) .  \label{9}
\end{equation}%
Let $\beta ,\gamma $ be two real numbers such that $0<\beta <\min \left\{
1,C_{5}\right\} $, where $C_{5}$ is given in Corollary \ref{cor6}, and $%
\gamma >\delta $ $(\gamma >1)$ satisfies $\gamma ^{p^{-}}\left\Vert
b\right\Vert _{1}>1$. If we use (\ref{9}), we have $F(x,t)\leq F(x,0)=0$ for 
$t\in \left[ 0,\beta \right] $, and%
\begin{equation}
\int\limits_{\partial \Omega }\sup_{0\leq t\leq \beta }F(x,t)d\sigma \leq
\int\limits_{\partial \Omega }F(x,0)d\sigma =0.  \label{10}
\end{equation}%
Moreover, due to $\gamma >\delta $ and (\ref{9}) we obtain $%
\int\limits_{\partial \Omega }F(x,\delta )d\sigma >0$ and 
\begin{equation}
\frac{1}{C_{5}^{p^{+}}}\frac{\beta ^{+}}{\gamma ^{p^{-}}}\int\limits_{%
\partial \Omega }F(x,\delta )d\sigma >0.  \label{11}
\end{equation}%
If we consider (\ref{10}) and (\ref{11}), then we get%
\begin{equation*}
\int\limits_{\partial \Omega }\sup_{0\leq t\leq a}F(x,t)d\sigma \leq 0<\frac{%
1}{C_{5}^{p^{+}}}\frac{\beta ^{+}}{\gamma ^{p^{-}}}\int\limits_{\partial
\Omega }F(x,\delta )d\sigma .
\end{equation*}%
Define $u_{0},u_{1}\in W_{a,b}^{1,p(.)}\left( \Omega \right) $ with $%
u_{0}(x)=0$ and $u_{1}(x)=\gamma $ for any $x\in \Omega .$ If we take $r=%
\frac{1}{p^{+}}\left( \frac{\beta }{C_{5}}\right) ^{p^{+}}$, then $r\in
\left( 0,1\right) ,$ $\Phi (u_{0})=\Psi (u_{0})=0$ and 
\begin{eqnarray*}
\Phi (u_{1}) &=&\int\limits_{\Omega }\frac{1}{p(x)}b(x)\gamma ^{p(x)}dx\geq 
\frac{\gamma ^{p^{-}}}{p^{+}}\int\limits_{\Omega }b(x)dx=\frac{1}{p^{+}}%
\gamma ^{p^{-}}\left\Vert b\right\Vert _{1} \\
&\geq &\frac{1}{p^{+}}>r.
\end{eqnarray*}%
Thus we have $\Phi (u_{0})<r<\Phi (u_{1})$ and%
\begin{equation*}
\Psi (u_{1})=-\int\limits_{\partial \Omega }F(x,u_{1})d\sigma
=-\int\limits_{\partial \Omega }F(x,\gamma )d\sigma <0.
\end{equation*}%
Then \textit{(ii)} of Theorem \ref{teo9} is verified.

On the other hand, we have%
\begin{eqnarray*}
-\frac{\left( \Phi (u_{1})-r\right) \Psi (u_{0})+\left( r-\Phi
(u_{0})\right) \Psi (u_{1})}{\Phi (u_{1})-\Phi (u_{0})} &=&-r\frac{\Psi
(u_{1})}{\Phi (u_{1})} \\
&=&r\frac{\int\limits_{\partial \Omega }F(x,\gamma )d\sigma }{%
\int\limits_{\Omega }\frac{1}{p(x)}b(x)\gamma ^{p(x)}dx}>0.
\end{eqnarray*}%
Now, we consider the case $u\in W_{a,b}^{1,p(.)}\left( \Omega \right) $ with 
$\Phi (u)\leq r<1.$ Due to $\frac{1}{p^{+}}I(u)\leq \Phi (u)\leq r$, we have 
\begin{equation*}
I(u)\leq p^{+}r=\left( \frac{\beta }{C_{5}}\right) ^{p^{+}}<1.
\end{equation*}%
By Proposition \ref{pro 4} we get $\left\Vert u\right\Vert _{1,p(.),a,b}<1$
and 
\begin{equation*}
\frac{1}{p^{+}}\left\Vert u\right\Vert _{1,p(.),a,b}^{p^{+}}\leq \frac{1}{%
p^{+}}I(u)\leq \Phi (u)\leq r.
\end{equation*}%
If we consider Corollary \ref{cor6}, then we get 
\begin{equation*}
\left\vert u(x)\right\vert \leq C_{5}\left\Vert u\right\Vert
_{1,p(.),a,b}\leq C_{5}\left( p^{+}r\right) ^{\frac{1}{p^{+}}}=\beta
\end{equation*}%
for all $u\in W_{a,b}^{1,p(.)}\left( \Omega \right) $ and $x\in \Omega $
with $\Phi (u)\leq r.$

This follows that%
\begin{equation*}
-\underset{u\in \Phi ^{-1}\left( \left( -\infty ,r\right] \right) }{\inf }%
\Psi (u)=\underset{u\in \Phi ^{-1}\left( \left( -\infty ,r\right] \right) }{%
\sup }-\Psi (u)\leq \int\limits_{\partial \Omega }\sup_{0\leq t\leq \beta
}F(x,t)d\sigma \leq 0.
\end{equation*}%
Then we have%
\begin{equation*}
-\underset{u\in \Phi ^{-1}\left( \left( -\infty ,r\right] \right) }{\inf }%
\Psi (u)<r\frac{\int\limits_{\partial \Omega }F(x,\gamma )d\sigma }{%
\int\limits_{\Omega }\frac{1}{p(x)}b(x)\gamma ^{p(x)}dx}
\end{equation*}%
and%
\begin{equation*}
\underset{u\in \Phi ^{-1}\left( \left( -\infty ,r\right] \right) }{\inf }%
\Psi (u)>\frac{\left( \Phi (u_{1})-r\right) \Psi (u_{0})+\left( r-\Phi
(u_{0})\right) \Psi (u_{1})}{\Phi (u_{1})-\Phi (u_{0})}.
\end{equation*}%
Thus condition \textit{(iii)} of Theorem \ref{teo9} is obtained. This
completes the proof.
\end{proof}

\end{document}